\documentclass[12pt]{amsart}

\RequirePackage[nonfrench]{kotex}
\usepackage{kotex}

\usepackage{epsfig}
\usepackage{amsfonts}
\usepackage{amssymb}
\usepackage{amsmath}
\usepackage{amsthm}
\usepackage{amscd}
\usepackage{amstext}
\usepackage{latexsym}
\usepackage{array}
\usepackage{graphicx, subcaption}
\usepackage{todonotes}

\usepackage{xy}

 \usepackage[backref,bookmarks=false]{hyperref}

\input xy
\xyoption{all}

\usepackage{enumerate}

\newcommand{\ds}{\displaystyle}

\newcommand{\ZZ}{\mathbf{Z}}
\newcommand{\CC}{\mathbf{C}}
\newcommand{\QQ}{\mathbf{Q}}
\newcommand{\RR}{\mathbf{R}}

\newcommand{\PP}{\mathbf{P}}
\newcommand{\JJ}{\mathcal{J}}

\newcommand{\OO}{\mathcal{O}}

\newcommand{\iddb}{i \partial \overline{\partial}}

\newcommand{\db}{\overline{\partial}}
\newcommand{\al}{\alpha}

\newcommand{\qa}{\quad}
\newcommand{\vp}{\varphi}

\newcommand{\mfa}{\mathfrak{a}}

\newcommand{\xr}{ X_{\reg}}

\newcommand{\yr}{ Y_{\reg}}

\newcommand{\Wr}{W_{\reg}}

\newcommand{\zr}{ Z_{\reg}}

\newcommand{\noi}{\noindent}

\providecommand{\wt}[1]{\widetilde{#1}}

\providecommand{\abs}[1]{\left|#1\right|}

\providecommand{\norm}[1]{\lVert#1\rVert}

\theoremstyle{plain}

\newtheorem{theorem}{Theorem}[section]

\newtheorem{prop}[theorem]{Proposition}

\newtheorem*{theorem*}{Theorem}
\newtheorem*{lemma*}{Lemma}

\newtheorem{lemma}[theorem]{Lemma}
\newtheorem{corollary}[theorem]{Corollary}
\newtheorem{proposition}[theorem]{Proposition}
\newtheorem{definition}[theorem]{Definition}

 \newtheorem{example}[theorem]{\textnormal{\textbf{Example}}}

\theoremstyle{remark}
\newtheorem{remark1}[theorem]{Remark}

\DeclareMathOperator{\divisor}{div}

\DeclareMathOperator{\reg}{reg}

\DeclareMathOperator{\Div}{div}

\DeclareMathOperator{\sing}{sing}

\DeclareMathOperator{\ord}{ord}

\begin{document}

\title[$L^2$ extension for log canonical pairs]{$L^2$ extension of holomorphic functions \\ for log canonical pairs}


\subjclass[2010]{} 

\dedicatory{In memory of Jean-Pierre Demailly}

\author{Dano Kim}

\date{}

\maketitle

\begin{abstract}

\noindent
In a general $L^2$ extension theorem of Demailly  for log canonical pairs, the $L^2$ criterion with respect to a measure called   the Ohsawa measure determines when a given holomorphic function can be extended. 
Despite the analytic nature of the Ohsawa measure, we establish a geometric characterization of this analytic criterion using  the theory of log canonical centers from algebraic geometry. Along the way, we characterize when the Ohsawa measure fails to have generically smooth positive density, which answers an essential question arising from Demailly's work.

\end{abstract}

\section{Introduction}

  Let $Y \subset X$ be a submanifold of a complex manifold. Let $L$ be a holomorphic line bundle on $X$ and $K_X$  the canonical line bundle of $X$.  An \emph{$L^2$ extension theorem} says:  (under suitable conditions on $X, Y, L, \ldots$)  if a certain $L^2$ norm $ \norm{s}_Y$   is finite for a holomorphic section $s$ on $Y$ of  $(K_X \otimes L)|_Y$, then there exists  $\tilde{s} \in H^0(X, K_X \otimes L)$   such that  $\tilde{s}|_Y = s \; \text{ and } \; \norm{\tilde{s}}_X \le c \norm{s}_Y $ for some constant $c>0$.

 Since \cite{OT87}, there have been extensive developments on $L^2$ extension theorems (e.g. cf. \cite{O2}, \cite{M93}, \cite{O4}, \cite{O5},  \cite{D00}, \cite{S02}, \cite{MV}, \cite{K07}, \cite{GZZ12}, \cite{B13}, \cite{GZ14}, \cite{Ca17}, \cite{CDM17}, \cite{O20} and many others).   $L^2$ extension theorems are not only analogous to vanishing and injectivity theorems in complex algebraic geometry, but also often serve as their generalizations (cf. \cite{S02}). In view of the central importance of these powerful tools, it is crucial to understand the criterion of extension given by the above  finiteness of the norm  $\norm{s}_Y$. Such finiteness is nontrivial to achieve due to singularity of the measure against which the norm is taken. The purpose of this paper is to understand such singularities.

 Let $\Psi$ be a quasi-psh function with analytic singularities (Definition~\ref{analyticsing}) on a complex manifold $X$. 
  Generalizing \cite[Thm. 4]{O5} (and other earlier results such as \cite{O2}, \cite{M93},  \cite{D00}), Demailly~\cite[Thm. 2.8, Thm. 2.12]{D15} established general $L^2$ extension theorems for the pair  $(X, \Psi)$.
 Given such $(X, \Psi)$, one can always scale the function so that the new pair $(X, c \Psi)$ for some $c >0$  is \emph{log canonical}, which is the standard generality of fundamental importance and many applications in algebraic geometry (cf. \cite{KM}) vastly generalizing the classical picture of differential forms with a pole along a  smooth hypersurface. From now on, we assume that $(X, \Psi)$ itself is a log canonical pair, as in \cite[Thm. 2.8]{D15}.   \footnote{Also the log canonical case includes the settings of \cite{M93},  \cite{D00} and \cite[Thm. 4]{O5}, cf. Remark~\ref{MD}. On the other hand, to our knowledge, in the generality of `without the log canonical condition', the $L^2$ extension result of \cite{D15} lacks a suitable quantitative $L^2$ estimates, cf. \cite[(2.15)]{D15}.  }
  
    In the $L^2$ extension theorem \cite[Thm. 2.8]{D15} for a log canonical pair,   the subvariety $Y \subset X$ is taken  to be the \emph{non-klt locus} $N(\Psi)$ of $(X,\Psi)$, i.e. the  subvariety defined by the multiplier ideal of $\Psi$.    
 Also the input norm $\norm{s}_Y $ in \cite[Thm. 4]{O5} is taken as the $L^2$ norm with respect to a measure called the Ohsawa measure ${dV[\Psi]}$, which is defined in terms of a certain limit which is taken over some tubular neighborhoods shrinking to the pole set of $\Psi$.

\begin{definition}\cite{O5}, \cite{D15} \label{Omeasure}
   Let $(X, \Psi)$ and $Y$ be as above. 
Let $dV_X$ be a smooth volume form on $X$. Let $\yr$ be the regular locus of $Y$, i.e. the set of nonsingular points. A positive measure $d\mu$ on $\yr$ is called  the  \textbf{\emph{Ohsawa measure}} $dV[\Psi]$ of $\Psi$ on $Y$ (with respect to $dV_X$) if it satisfies the following condition:  for every $g$, a real-valued compactly supported continuous function on $\yr$ and for every $\tilde{g}$, a compactly supported extension of $g$ to $X$, we have the relation 
 
 \begin{equation}\label{dvp}
 \int_{\yr} g \; d\mu = \lim_{t \to -\infty} \int_{\{x \in X, t < \Psi(x) < t+1 \}} \tilde{g} e^{-\Psi} dV_{X}. 
\end{equation}

\end{definition}

   In some sense, this analytic definition is naturally dictated by the framework of the $L^2$ estimates for $\db$ operator (cf. \cite{OT87}, \cite{O5}, \cite{D15}) so that the $\db$ equation for the extension problem can be solved for a given section (cf. around (5.17) of \cite{D15}). 
   
   The Ohsawa measure first appeared in \cite{O4} and later used in the important work \cite{O5} when $\Psi$ is locally of the form $\Psi (z) = k \log (\abs{z_1}^2 + \ldots + \abs{z_k}^2) + O(1) $ near $Y$ where $(z_1 = \ldots = z_k = 0)$ are local coordinate defining equations for a submanifold $Y$ of codimension $k$. 
   The Ohsawa measure plays important roles in various recent work related to $L^2$ extension theorems, cf. \cite{C21}, \cite{De21}, \cite{GZ14}, \cite{K20}, \cite{KS20} and others. 
       Still, some essential understanding  has been missing. The following question arose from communications with Chen-Yu Chi~\cite{C20}. 
   \\

\noi \textbf{Question.} 
 \emph{ Let $(X, \Psi)$ be a log canonical pair. Let $Z$ be an irreducible component of the non-klt locus $Y$ of $(X, \Psi)$.   When does the Ohsawa measure $dV[\Psi]$ have smooth positive density on a nonempty Zariski open subset of $Z$? } 

\; 
 
 Here the density refers to the density function with respect to the Lebesgue measure in local coordinates. 
In fact, the assertion of the question was claimed in  \cite[Prop. 4.5 (a)]{D15} \footnote{Comparing with the Question, although `nonempty' and `each irreducible component of $Z_p$' are not spelled out in the statement of \cite[Prop. 4.5 (a)]{D15}, they are clearly indicated by its proof and also by conversations we had with the author of \cite{D15}.},        which would mean that the Ohsawa measure always has singularities only along a proper Zariski closed subset of $Z$.

 However, this claim was made under the implicit assumption that each irreducible component of the non-klt locus $Y$ is dominated by precisely one divisor $E \subset X'$ with \emph{discrepancy}  $a(E, X, \Psi)$  equal to $-1$, in a log resolution $f : X' \to X$ of the log canonical pair $(X, \Psi)$.

  We will call such a divisor $E$ on $X'$ as a  \emph{log canonical place} for $(X, \Psi)$ and its image $f(E)$ on $X$  a \emph{log canonical center} for $(X, \Psi)$. Log canonical places are exactly those divisors on $X'$ that contribute to the nontriviality of the multiplier ideal of $\Psi$.   See \S 2.1, 2.2 for more details on these notions. 
  Also, a log canonical center is said to be \emph{maximal} if it is not properly contained in another log canonical center. Maximal log canonical centers are nothing but the irreducible components of the non-klt locus of $(X, \Psi)$.

 A crucial subtlety with log canonical pairs   is that,  given a log canonical center $Z$,  there may exist two or more log canonical places for $Z$  on a log resolution $f: X' \to X$.     Suppose that there exist at least two  log canonical places $E_1, \ldots, E_m$ for a maximal log canonical center $Z$. Chen-Yu Chi~\cite{C20} observed that, if  some $E_i$ and $E_j$ ($i\neq j$) happen to have the image $f(E_i \cap E_j)$ dominate $Z$, then the Ohsawa measure cannot have smooth density on a nonempty Zariski open subset of $Z$ (cf. the proof of Theorem~\ref{infinite}). 
 
 Indeed, for example, suppose that $Z$ is a point having two lc places $E_1, E_2$ in a log resolution (as in Example~\ref{Fujino}).  In this special case, a priori, the Ohsawa measure $dV[\Psi]$ (restricted to $Z$) is represented by a single nonnegative real number including infinity. Then, as in Lemma~\ref{demu}, the number $dV[\Psi]$ is equal to the sum of the integral of a measure with poles along $(E_1 + F)|_{E_2}$ over $E_2$ and the integral of a measure with poles along $(E_2 + F)|_{E_1}$ over $E_1$ (for an appropriate divisor $F$), both of which are infinities.

  In order to resolve this issue, we answer the above Question by showing that whenever $Z$ has non-unique log canonical places, the Ohsawa measure has singularities along the entire $Z$. 
  
\begin{theorem}\label{infinite}
 Let $(X, \Psi)$ be a log canonical pair where $X$ is a complex manifold and $\Psi$ a quasi-psh function on $X$ with analytic singularities. 
  Let $Z$ be a maximal log canonical center of the pair $(X, \Psi)$. 
 If $Z$ has at least two log canonical places in a log resolution, then the Ohsawa measure $dV[\Psi]_Z$ is the infinity measure, i.e. $\ds \int_{\zr} g \; dV[\Psi]_Z = \infty$  for every continuous function $g : \zr \to \RR$ unless $g \equiv 0$. 

\end{theorem}

 Here  $dV[\Psi]_Z$ is the restriction of the Ohsawa measure to $Z$, see Definition~\ref{iber}. Theorem~\ref{infinite}  says that, in concrete applications of $L^2$ extension where one is given $Z$ and needs to find $\Psi$ so that $Z$ is a maximal log canonical center of the pair $(X, \Psi)$, one would prefer to avoid such $\Psi$ as in the above theorem. Sometimes one can use \emph{tie-breaking}
which would enable one to modify $\Psi$ in order to avoid the situation of non-unique lc places, cf. \cite{Ko97}, \cite{Ka97}, \cite[\S 8.1]{Ko07}. However this will require some strict positivity, which is 
  not available under the general curvature condition of Theorem~\ref{extension}.  This makes it even more  important to recognize the kind of limitation for $L^2$ extension as described in Theorem~\ref{infinite}.

 Now we turn to our next main result. As mentioned earlier,  the $L^2$ norm condition with respect to the Ohsawa measure gives the analytic criterion for extension in Demailly's theorem~\cite{D15}. 
   However the condition will be difficult to check using the definition of the Ohsawa measure directly, as long as the singularity of the Ohsawa measure remains  mysterious. 
  We  establish the following  geometric characterization of the $L^2$ norm condition in terms of log canonical centers. 
 
\begin{theorem}\label{main1} 
Let $(X, \Psi)$ be a log canonical pair where $X$ is a complex manifold and $\Psi$ a quasi-psh function on $X$ with analytic singularities (Definition~\ref{analyticsing}). Let $Y \subset X$ be the non-klt locus of $(X, \Psi)$. 
Let $s$ be a holomorphic function on $Y$.  
Then the following conditions (1) and (2) are equivalent:

(1) The function  $s$ is locally $L^2$ at every point of $Y$ with respect to the Ohsawa measure ${dV[\Psi]}$.  

(2) The function $s$ vanishes along 
 \begin{itemize}
 \item
all non-maximal log canonical centers of $(X, \Psi)$ and
\item   all maximal log canonical centers of $(X, \Psi)$ with non-unique log canonical places. 
\end{itemize}
\end{theorem}

Note that the Ohsawa measure is defined on the regular locus $\yr$ of $Y$ and all the integrals are 
taken on $\yr$ while  we take every point of $Y$ in the condition (1).  

Theorem~\ref{main1} in particular says that, when $X$ is compact, the  geometric condition in the above (2) alone determines whether a given section on $Y$ can be extended using Demailly's theorem \cite[Thm. 2.8]{D15} (under its curvature conditions).

Recall that the definition of the Ohsawa measure is essentially dictated by the framework of the $L^2$ estimates for $\db$ operator : a priori, there is no immediate reason to foresee that the theory of log canonical centers (cf. \cite{Ka97}) from algebraic geometry will be useful to understand the Ohsawa measure as in Theorem~\ref{main1}.

 The methods of proofs of Theorems~\ref{infinite} and \ref{main1} consist of combining the analytic theory of the Ohsawa measure (cf. \cite{D15}, \cite{O5}) and the algebro-geometric theory of log canonical centers (cf. \cite{Ka97}, \cite{A03}, \cite{Ko07}, \cite{F11}). The starting point of this combination is Demailly's idea of using log resolutions of pairs $(X, \Psi)$ to understand the Ohsawa measure, which can be viewed as the direct image of some other measures on the log resolution (Lemma~\ref{demu}). We then analyze the geometry of log canonical places where a crucial role  is played by a connectedness result,  Proposition~\ref{basic0}.   
 In view of these methods (and also the roles played in formulating the right statements),  our main results Theorem~\ref{infinite} and Theorem~\ref{main1} can be regarded as interesting applications of algebraic geometry to complex analysis.

As a consequence of Theorem~\ref{main1}, we have the following characterization of the local integrability of the Ohsawa measure. 

\begin{corollary}[=Corollary~\ref{coro2} (2)]\label{Coro1.3}
 Let $(X, \Psi)$ be a log canonical pair as in Theorem~\ref{main1}. The Ohsawa measure ${dV[\Psi]}$ is locally integrable if and only  if each connected component $Z$ of the non-klt locus $Y$ is irreducible and $Z$ is a   minimal  log canonical center of $(X, \Psi)$ with a unique log canonical place.

\end{corollary}

This means that $L^2$ extension theorems can indeed extend as many sections as would be expected from algebraic geometry (cf. \cite{KM}, \cite{Ka97}), since Corollary~\ref{Coro1.3} says that $dV[\Psi]$ presents no local obstruction for extension when $Z$ is a minimal log canonical center. 

As another consequence of Theorem~\ref{main1}, we  derive a version of $L^2$ extension theorem, namely Theorem~\ref{extension} (see the statement in Section 3).  It extends a section given on one irreducible component (i.e. a maximal log canonical center) of the non-klt locus, rather than given on the entire non-klt locus. This  provides significant flexibility in $L^2$ extension, as would be often needed  in applications.  \footnote{The author has a previous result of $L^2$ extension theorem \cite[Thm. 4.2]{K07} in such a setting, which is independent of the approach using the Ohsawa measure. More comparison between these results will be treated elsewhere. }

For Theorem~\ref{extension}, the following yields a geometric characterization of the input norm condition. 

\begin{corollary}\label{coro3}

Let $(X, \Psi)$ be a log canonical pair as in Theorem~\ref{main1}.  Let $Y$ be a maximal lc center of $(X, \Psi)$ with a unique lc place.  Let $s$ be a holomorphic function on $Y$. 
Then the following conditions (1) and (2) are equivalent: 

(1) The function  $s$ is locally $L^2$ at every point of $Y$ with respect to the Ohsawa measure ${dV[\Psi]}_Y$.  

(2) The function $s$ vanishes along all log canonical centers of $(X, \Psi)$ that are properly contained in $Y$. 

\end{corollary}

 A special case of Corollary~\ref{coro3} (and Theorem~\ref{main1}) is when 
 $Y$ is the only maximal lc center, i.e. when the non-klt locus of $(X, \Psi)$ is irreducible. 
  Note that, given one $Y$, there may exist  different statements  of $L^2$ extension (from different  $\Psi$'s) depending on which sets of subvarieties of $Y$ appear as properly contained log canonical centers of $(X, \Psi)$.

This article is organized as follows. In Section 2, we recall the theory of singularity of pairs and log canonical centers in our setting. We intend this section to be mostly self-contained for the convenience of readers whose primary interest is in analysis.    In Section 3, we revisit the theory of the Ohsawa measure and  also derive an $L^2$ extension theorem for a maximal log canonical center from \cite{D15} using Theorem~\ref{main1}. In Section 4, based on preparations from the previous sections, we give the proofs of the main results.

\qa

\noi \textbf{Acknowledgments.} 
The author would like to thank C.-Y. Chi for sharing his important observations and questions and O. Fujino for kindly providing Example~\ref{Fujino}. 
  The author would like to thank   J.-P. Demailly, T. Ohsawa, J. Kollár  and S. Boucksom for valuable and helpful conversations. 
  This research was supported by SRC-GAIA through NRF Korea grant No.2011-0030795 and by Basic Science Research Program through NRF Korea funded by the Ministry of Education (2018R1D1A1B07049683).

\qa

\section{Log canonical centers}

 We refer to \cite{D11} for the basic notions of  plurisubharmonic (psh for short) functions and  singular hermitian metrics of line bundles.  
 We  have some \textbf{{terminology and conventions}} used in this paper.

\begin{itemize}

\item

  We will often denote a singular hermitian metric $e^{-\vp}$ of a line bundle  simply by $\vp$ so that we can write additively both line bundles and  metrics as in $(L_1 + L_2 , \vp_1 + \vp_2)$.

\item

    A \emph{psh metric} is a singular hermitian metric of $L$ with semipositive curvature current (so that its local weight functions can be taken as psh  functions).  

\item

 Whenever integration is taken on a variety as in $\int_X$, it is taken on its regular locus as in $\int_{\xr}$. 

\item
A function on a complex manifold with values in $\RR \cup \{ -\infty \}$ is \emph{quasi-psh} if it is locally the sum of a psh function and a smooth $\RR$-valued function.

\end{itemize}

 \begin{definition}  \label{analyticsing} cf.  \cite[Def. 2.2]{D15}, \cite[Def. 9.1]{B20}
 A quasi-psh function on a complex manifold has \textbf{\emph{analytic singularities}} if there exists a real number $c>0$ and a coherent ideal sheaf $\mathfrak{a} \subset \OO_X$ such that locally we have
 
 \begin{equation}
  \vp = c \log \sum^{m}_{i=1} \abs{g_i}^2 + u 
  \end{equation} where $g_1, \ldots, g_m$ are local generators of $\mfa$ and $u$ a real-valued  smooth function. \footnote{Such a quasi-psh function was said to have neat analytic singularities in \cite{D15}. }  In this case, we will say that $\vp$ has singularities of type $\mfa^c$. 
 A psh metric of a line bundle is said to have \emph{analytic singularities} if its local weight functions  are psh with analytic singularities. 
 
 \end{definition}

    Let $L$ be a $\QQ$-line bundle on $X$. Let $D$ be an effective $\QQ$-divisor such that $D$ is $\QQ$-linearly equivalent to $L$. Note that $D$ is the  divisor $\frac{1}{m} \divisor (s_D)$ of a holomorphic section $s_D \in H^0 (X, mL)$ for some $m \ge 1$ such that $mL$ is a genuine holomorphic line bundle  on $X$. The section $s_D$ defines a psh metric of $L$ which can be   denoted by  ${\abs{s_D}}^{-\frac{2}{m}}$.

\subsection{Singularities of pairs}
 
 In the minimal model program and birational geometry~\cite{KM}, singularity of pairs such as $(X,D)$ plays an important and central role, cf. \cite{Ko97}, \cite{Ko13}. In this paper, we will work with singularity of pairs where $X$ is a complex manifold and $D$ is replaced by a psh metric or by a quasi-psh function.   
  
  \begin{definition}\label{Psi}
 Let $X$ be a complex manifold.  A \textbf{\emph{pair}} $(X, \psi)$  will refer to a pair of $X$ and a psh metric $e^{-\psi}$ with analytic singularities of a $\QQ$-line bundle $L$ on $X$.   Similarly, a pair $(X, \Psi)$ will refer to a pair of $X$ and  $\Psi$, a quasi-psh function with analytic singularities on $X$. 
 
  \end{definition}

 In the setting of the first sentence, we will  say that $(X, \psi)$ is associated to the $\QQ$-line bundle $K_X + L$.  
 We can also view $\psi$ as represented by a quasi-psh function (say $\Psi$) once we have a choice of a $C^{\infty}$ hermitian metric $h$ of $L$ :  i.e. define $\Psi$ by the relation  
$ e^{-\psi} = h e^{-\Psi}$ on $X$.   
 
 In the rest of this section, we introduce the fundamental notions from singularity of pairs (cf. \cite{KM}, \cite{Ko97}, \cite{Ka97}) in the setting of $(X, \Psi)$.  Since $\Psi$ has analytic singularities (say, of type $\mfa^c$),  these notions are well known in the setting of pairs taken with ideal sheaves (with formal exponents) $(X,  \mfa^c)$, see \cite[\S 8.1]{Ko13}.

  For simplicity of notation, we will discuss these fundamental notions mostly for pairs $(X, \Psi)$ (among the two in Definition~\ref{Psi}), but one can also put $(X, \psi)$ in the place of $(X, \Psi)$ (whenever it is clear).

 First, we will say that a quasi-psh function $\vp$ (and similarly for a psh metric) has \textbf{poles along} an effective ($\QQ$-)divisor $D$ if we have locally $\vp = \sum^m_{i=1} a_i \log \abs{g_i}^2  + u $ where $u$ is (bounded) $C^{\infty}$ and $\sum^m_{i=1} a_i \Div(g_i) $ is a local expression of the divisor $D$. 
 
 A \textbf{log resolution} for a quasi-psh function $\vp$ with analytic singularities of type $\mfa^c$ (and similarly for a psh metric with analytic singularities) is a smooth modification  (cf. \cite[II (10.1)]{DX})  $f: X' \to X$ such that the pullback $f^* \vp$ has poles along a divisor $G$ on $X'$ and $G+F$ is an snc (i.e. simple normal crossing) divisor where $F$ is the exceptional divisor of $f$. Such a log resolution exists since, as is well known, a principalization of the ideal sheaf $\mfa$ exists by Hironaka.  
 
 Now let $(X, \Psi)$ be as in Definition~\ref{Psi}. 
 Let $E \subset X'$ be a prime divisor. The \textbf{discrepancy} of $E$ with respect to $(X, \Psi)$ is defined by  $ a(E, X, \Psi) :=  \ord_E (K_{X'/X}) - \ord_E (\Psi)$ where $\ord_E (\Psi)$ is the generic Lelong number of $f^* \Psi$ along $E$, cf. \cite[\S 10]{B20}.   
 We have the following standard definition in this setting (cf. \cite{Ko97}, \cite{KM}). 

\begin{definition}\label{lcpair}

 Let $(X, \Psi)$  be as in Definition~\ref{Psi}. 
A pair $(X, \Psi)$  is lc  (resp. klt) if  $a(E, X, \psi) \ge -1$ (resp. $a(E, X, \psi) > -1$) for every prime divisor $E \subset X'$ in some (and hence, any \footnote{Cf. \cite{Ko97}, \cite{Ko13}}) log resolution $X' \to X$. 

\end{definition}

 We find it  convenient to indicate the discrepancies $a(E, X, \Psi)$ for $E$'s appearing in a log resolution by formally mimicking the standard notation for pairs $(X,D)$:

 \begin{equation}\label{discrep}
  K_{X'}  \equiv f^* (K_X + [\Psi])  + \sum_{E \subset X'} a(E, X, \Psi) E
 \end{equation}

\noi which generalizes the usual expression $ K_{X'}  \equiv f^* (K_X + D)  + \sum_{E \subset X'} a(E, X, \Psi) E$ when $\Psi$ has divisorial poles along an effective $\QQ$-divisor $D$ (cf. \cite[Notation 2.26]{KM}) where the equality is understood as numerical equivalence of divisors. In \eqref{discrep}, we  regard the notation $\equiv$ and $[\Psi]$ as purely formal.


\begin{remark1}

  In the case of a complex manifold $X$, in general one does not have a global canonical divisor $K_X$ which corresponds to the canonical bundle. However, as is well known, the notion of discrepancies (as defined below) is still well-defined, cf. \cite[\S 3]{F22}.

\end{remark1}

 \subsection{Log canonical centers}

 Next, we  recall the theory of  log canonical (lc) centers from algebraic geometry, which  can be regarded as  providing more refined information than  multiplier ideals alone.

  Let $(X,\Psi)$ be an lc  pair. We will say that $\Psi^{-1} (-\infty)$ is  the {pole set} of $\Psi$.   Define the \textbf{non-klt locus} of $(X, \Psi)$ to be the union of the images of those prime divisors $E$ on a log resolution $X'$ with discrepancy $a(E, X, \Psi) = -1$. The non-klt locus is the underlying set of the reduced subspace associated to the multiplier ideal $\JJ(\Psi)$, hence well defined, independent of the choice of a log resolution.   It is clear that the non-klt locus is a (possibly strict) subset of the pole set.

  \begin{example}

 Let $X = \CC^2$. Let $\Psi = \al \log \abs{x}^2 + \beta \log (\abs{x}^2 + \abs{y}^2)$. If $0 < \al <1$ and $\al + \beta =2$, the non-klt locus of $(X, \Psi)$ is equal to the origin, which is strictly contained in the pole set, the line $x = 0$. 

\end{example}

  A \textbf{log canonical center} (or an \textbf{lc center}) of $(X,\Psi)$ is an irreducible subvariety $Y \subset X$ that is the image of a prime divisor $E$ in a log resolution $X' \to X$ as above, with its discrepancy $a(E, X, \Psi)$ equal to $-1$ (i.e. the lowest possible value for an lc pair) on a log resolution of the pair $(X,\Psi)$. We will call such $E$ an \textbf{lc place} of $Y$, following \cite{Ka97}. \footnote{In \cite[Def. 1.3]{Ka97}, an lc place refers to an equivalence class of such $E$ determined by strict transform between different log resolutions. This convention is useful, but  in this paper, it is enough to  regard an  lc place as a prime divisor appearing in a  particular log resolution. }

  Given a pair,  when $Y$ is an lc center and there is no other lc center $Y_1$ such that $ Y_1 \supsetneq Y $, we call  $Y$ a \textbf{maximal lc center}.  The maximal lc centers are precisely the irreducible components of the non-klt locus of $(X, \Psi)$. 
 On the other hand, an lc center $Y$ is called \textbf{minimal} if there are no other lc centers properly contained in $Y$.

\begin{example}

 Let $Y \subset X$ be an irreducible smooth subvariety in a smooth projective variety. Then there exists an ample line bundle $L$ with a psh metric $\Psi$ such that $Y$ is a minimal lc center of the pair $(X, \Psi)$. 

\end{example}

\begin{example}\label{xyz}
 
  Let $X = \CC^3$ with coordinates $(x,y,z)$. Let $Y$  be the hyperplane given by $(z=0)$ and let $Z$ be the line given by $(x=y=0)$. Define $\Psi = \log \abs{z}^2 + \frac{2}{3} \log (\abs{x}^2 + \abs{y}^2)$ which is a psh function and also a psh metric for the trivial line bundle $L = \OO_X$. Then the pair $(X, \Psi)$ is lc and the lc centers are $Y$, $Z$ and the point $p := Y \cap Z$.  Note that the maximal lc centers are $Y$ and $Z$ while $p$ is the only minimal lc center of $(X, \Psi)$. 

\end{example}

\begin{example}\label{sncdiv}

  Let $D = D_1 + \ldots + D_m$ be an snc divisor ($m \ge 2$) with coefficients $1$ on a complex manifold $X$. Let $\Psi$ be a quasi-psh function having poles along $D$. Then $(X, \Psi)$ is lc. A connected component $C$  of $D_i \cap D_j$ ($i \neq j$) is irreducible and is an lc center of $(X, \Psi)$. 

\end{example}

 The following says that the property of a maximal lc center having a unique lc place is intrinsic, i.e. independent of the choice of a log resolution. 

\begin{proposition}\label{elc} 

 Let $Y$ be a maximal lc center of an lc pair $(X, \Psi)$. If there exists a log resolution $X' \to X$ of $(X, \Psi)$ such that $Y$ has a unique lc place $E$ on $X'$, then there cannot exist another log resolution of $(X, \Psi)$ where $Y$ has at least two lc places. 
\end{proposition}

\begin{proof} 
  
  Since the assertion is local, we may assume that $E_0 (:= E), E_1,  \ldots, E_m$ are all the lc places (associated to some lc centers) of $(X, \Psi)$ in the given log resolution $f: X' \to X$.  Suppose that there exists another log resolution $h$ where the lc center $Y$ has at least two lc places.  We may assume that it factors through $f$ so that it is given as the composition $h = f \circ g$ of $g: X'' \to X'$ and $f: X' \to X$. 
  
  Let $G \subset X''$ be an lc place of $Y$ that is genuinely different from $E$, i.e. not a strict transform of $E$. It is known that (cf. \cite[Lem. 1.11.14]{Ka14}) the center of $G$ on $X'$, i.e. $g(G)$, must be an irreducible component of some intersection among $E_0, \ldots, E_m$. Since $G$ is not a strict transform of $E$,   there should exist at least one $E_j \; (j \ge 1)$ such that $g(G) \subset E_j$. Then it is impossible to have $h(G) = Y$   since we have $h(G) \subset f(E_j)$ and $Y$ is a maximal lc center, which forces $f(E_j) = Y$. This contradicts to the condition that $E$ is the unique lc place of $Y$ on $X'$. 
\end{proof}

  We remark that the assertion of Proposition~\ref{elc} does not hold when $Y$ is not maximal. For example, let $(X, Y_1 + Y_2)$ be an snc pair with $\dim X = 2$ such that $Y_1$ and $Y_2$ intersects at $p \in X$ transversally. The blow up of $p$ provides a log resolution with a unique lc place $E$ for the lc center $p$. Further blow ups of the intersection between $E$ and the strict transforms of $Y_1$ and  $Y_2$ provide log resolutions with non-unique lc places for $p$. 
 
  An example of a maximal lc center with non-unique lc places is as follows, which we learned from Osamu Fujino.

\begin{example}[O. Fujino]\label{Fujino}

Let $X = \PP^2$ and let $C$ be a smooth quadric curve on $X$. Let $p \in C$ be a point and $L$ the tangent line of $C$ at $p$. Let $L_1$ and $L_2$ be two general lines passing through $p$. 
Now define an effective $\QQ$-divisor on $X$ by  $D := aC + b L + cL_1 +d L_2$ where $a,b,c,d \in \QQ \cap (0,1)$. 

Let $f: Y \to X$ be the blow up of $p$ with the exceptional divisor $E \subset Y$. There is a point $q \in E$ where the strict transforms of $C$ and $L$ intersect. Let $g: W \to Y$ be the blow up of $q$ with the exceptional divisor $G \subset W$. Then $h = f \circ g : W \to X$ is a log resolution of $(X,D)$ and we have $K_W + D_W = h^* (K_X + D)$ where $D_W$ has coefficients $3a + 2b + c + d -2$ for $G$ and $2a + b+ c+ d -1$ for $E'$ (the strict transform of $E$). When $a+b = 1$ and $a+c+d = 1$, the following holds: the pair $(X,D)$ is lc but not klt at $P$. The non-klt locus is exactly $P$, hence the lc center $P$ is both maximal and minimal. It has two lc places $G$ and $E'$. 
 
\end{example}

  The following connectedness result has been known from \cite[Thm. 1.6]{Ka97} (cf. \cite{A03}, \cite{A11}, \cite[Lem. 2.3]{F11}).

\begin{proposition}\label{basic0}

 Let  $(X, \Psi)$ be an lc pair.  Let $f: X' \to X$ be a log resolution of $(X, \Psi)$. 
 Let $W$ be the union of a finite set of lc centers of $(X, \Psi)$ equipped with the reduced structure. Assume that $f^{-1} (W)$ is of codimension $1$.  Let $T$ be the union of prime divisors on $X'$ given by the lc places of $(X, \Psi)$ whose images are contained in $W$. Then the restriction $f: T \to W$ has connected fibers. 

\end{proposition}   
  
\begin{proof} 

 These are special cases for $X$ smooth of \cite[(4.4)]{A03}, \cite[(4.2)]{A11}, \cite[(2.3)]{F11}, with the obvious modifications in using relative vanishing theorems for $f$ now being a projective bimeromorphic morphism between complex manifolds.  
\end{proof}

\begin{corollary}\label{conf}

  Let $Y$ be an lc center of an lc pair $(X, \Psi)$.  If a log resolution $f: X' \to X$ of $(X, \Psi)$ contains a unique lc place $E \subset X'$ for $Y$, then the restriction morphism
  $f|_E : E \to Y$ has connected fibers. 
\end{corollary} 

\begin{proof} 

 Let $T$ be the union of all the lc places on $X'$ whose images are contained in $Y$. Let $Z \subset Y$ be the closed analytic subset given as the union of those lc centers properly contained in $Y$. We see that $f: E \to Y$ has connected fibers over the Zariski open set $Y \setminus Z$ by Proposition~\ref{basic0}. Hence it also has connected fibers over $Y$ in view of Stein factorization. 
\end{proof}

 \begin{remark1}\label{Kollar19}
  In the case when a maximal lc center $Y$ has more than one lc places (say $E_1, E_2, \ldots, E_m$), there are examples of pairs (cf. Koll\'ar~\cite{Ko19}) where the morphism $E_j \to Y$ is not with connected fibers for some $j $. 
 \end{remark1}

 The following is the key fundamental properties of lc centers, which we recall from \cite{Ka97}, \cite{A03}, \cite[Thm. 2.4]{F11} in the current setting. \footnote{See also a recent work of Fujino \cite[Thm. 7.1]{F23} in the generality of $X$ being complex analytic spaces.} 

\begin{proposition}\label{basic}

 Let  $(X, \Psi)$ be an lc pair. 
 
 \begin{enumerate}

\item

 Locally there are at most finitely many lc centers of $(X, \Psi)$. 

 \item
  The intersection of two lc centers can be written as a union of lc centers.

\item
Let $p \in X$ be a point in the non-klt locus of $(X, \Psi)$.  Then there exists a unique minimal lc center $C_p$ at $p$, i.e. minimal with respect to inclusion among those lc centers passing through $p$.  Also $C_p$ is normal at $p$. 
\end{enumerate}

\end{proposition}

 \begin{proof} 
 
  For the sake of convenience for readers, we recall some of the arguments (cf. \cite[Thm. 2.4]{F11}). (1) is clear from the existence of a log resolution. The first sentence of (3) follows from (2).  In the rest, we will show (2).

    Let $C_1$ and $C_2$ be two lc centers with nonempty intersection (which is not necessarily irreducible or connected). Let $p \in C_1 \cap C_2$. It suffices to show that there exists an lc center $C_3$ such that $p \in C_3 \subset C_1 \cap C_2$. Suppose that there does not exist such $C_3$ for given $p$. 
 Now apply Proposition~\ref{basic0} to the set $\{ C_1, C_2 \}$. Let $T = \sum_{j \in J} E_j$ where $J$ is a finite set. We may assume that the log resolution was chosen such that every lc center contained in the union $C_1 \cup C_2$ has at least one lc place occurring in $T$. 

  The image of each $E_j$ is an lc center contained in $C_1 \cup C_2$. Since we are assuming that none of them are both passing through $p$ and contained in $C_1 \cap C_2$, by restricting to an open neighborhood of $p$, we may assume that  the images $f(E_j)$ are all passing through $p$  and that $J$ is written as a disjoint union $J = J_1 \cup J_2$ such that if $i \in J_k$, then $f(E_i)$ is contained in $C_k$ but not in the intersection $C_1 \cap C_2$. 
 
 However, by the connectedness of the fiber $f^{-1} (p)$, there must exist $p \in J_1, q \in J_2$ such that $E_p$ and $E_q$ have nonempty intersection. This provides an lc place whose image is contained in $C_1 \cap C_2$, contradiction. 
 \end{proof}

  Note that if an lc center $Y$ of $(X, \Psi)$ is minimal in the sense given previously (i.e. there are no other lc centers properly contained in $Y$), then it is minimal at every point in $Y$. It follows that $Y$ is normal by (3) of the above proposition.

 \medskip

\section{$L^2$ extension theorems and the Ohsawa measure}
 
 In this section, we revisit the theory of the Ohsawa measure and  also derive an $L^2$ extension theorem for a maximal log canonical center from \cite{D15} using Theorem~\ref{main1}.

\subsection{Definition of the Ohsawa measure}

We first  revisit the definition and properties of the Ohsawa measure from \cite{D15}. We do this in a self-contained manner in the current log canonical setting, trying to provide explicit details for the Ohsawa measure (together with details prepared in \cite{KS20}).

 Let $(X, \Psi)$ be an lc (i.e. log canonical) pair as in Definition~\ref{lcpair} where $X$ is a complex manifold and $\Psi$ is a quasi-psh function with analytic singularities (i.e. neat analytic singularities in the sense of \cite{D15}, cf. Definition~\ref{analyticsing}). 
 Let $Y \subset X$ be a maximal lc center of the pair. 
  Let $\yr$ be the regular locus of $Y$, i.e. the set of regular points. 
  
  For our purpose in this paper, it is more natural to first define the Ohsawa measure on $Y$. One can see that the full Ohsawa measure (Definition~\ref{Omeasure}) on the entire non-klt locus is recovered from restrictions to each $Y$, see Proposition~\ref{recover}.

\begin{definition}  \cite[Section 2]{D15} \label{iber}
Let $dV_X$ be a smooth volume form on $X$.   The  \textbf{ \emph{Ohsawa measure}} $dV[\Psi]_Y$ of $\Psi$ on $Y$ (with respect to $dV_X$) is 
a positive measure $d\mu$ on $\yr$ satisfying the following condition:  for every $g$, a real-valued compactly supported continuous function on $\yr$ and for every $\tilde{g}$, a compactly supported extension of $g$ to $X$, we have the relation 
 
 \begin{equation}\label{dvp}
 \int_{\yr} g \; d\mu = \lim_{t \to -\infty} \int_{\{x \in X, t < \Psi(x) < t+1 \}} \tilde{g} e^{-\Psi} dV_{X}. 
\end{equation}

\end{definition}

If it exists, the Ohsawa measure is unique from basic properties of measures.   The existence is due to \cite[Prop.  4.5 (a)]{D15} in this generality (after the initial cases from \cite{O4},  \cite{O5} which was also used in \cite{GZ14}), whose arguments we will recall in our setting. 
 In its notation $dV[\Psi]_Y$,  its dependence on $dV_X$ is suppressed.

\begin{proposition}\label{exist}

 Let $Y$ and $(X, \Psi)$ be as above.  Then the Ohsawa measure $dV[\Psi]_Y$ of $(X, \Psi)$ on $Y$ exists. The measure $dV[\Psi]_Y$ has the property of  putting no mass on closed analytic subsets. 

 \end{proposition}
 
\begin{proof}

 Let $f: X' \to X$ be a log resolution of the pair $(X, \Psi)$.  Let $E_1, \ldots, E_m$ be the lc places of $Y$ and let $T := E_1 \cup \ldots \cup E_m$. As in \eqref{discrep}, we can indicate discrepancies (for a $\QQ$-divisor $F$ on $X'$) by writing
 \begin{equation}\label{discrep2}
 K_{X'} + E_1 + \ldots + E_m + F \equiv f^* (K_X + [\Psi]). 
 \end{equation}
 
   When we take the pull back of the RHS of \eqref{dvp} by $f$, we are in the position to consider the Ohsawa measure, say $d\nu$, of the quasi-psh function $f^* \Psi$ with respect to the volume form $f^* dV_X$ on $X'$ (which may have zeros along a divisor, cf. \cite[Section 3]{KS20}), i.e. the measure $d\nu$ satisfies

  \begin{equation}\label{dvp1}
\int_{T}  f^* g \; d\nu = \lim_{t \to -\infty} \int_{\{x \in X', \; t < f^* \Psi(x) < t+1 \}} (f^* \tilde{g}) \; e^{-f^* \Psi} f^* dV_{X}. 
\end{equation}

From  \cite[Prop.  4.5 (a)]{D15} (also see  \cite[Prop. 3.3 (2)]{KS20} for some more details spelled out), such $d\nu$ exists as a measure on the regular locus of $T$. (Note that $f^* dV_X$ may have zeros along a divisor, which does not matter for such existence.)  Moreover $d\nu$ is given by the collection of its restrictions to each $E_j$ or more precisely to each connected component of the regular locus of $T$, namely to each  

\begin{equation}\label{ejp}
E_j' := E_j \setminus ( \cup_{i \neq j} E_i)
\end{equation}
for $1 \le j \le m$.
 We will denote this restriction by $d\nu|_{E_j}$ (i.e. rather than $d\nu|_{E_j'}$). 
 By  \cite[Prop. 3.3 (2')]{KS20}, the restriction  $d\nu|_{E_j}$ has  poles along the divisor 

\begin{equation}\label{pole}
 (\sum_{i \neq j} E_i  + F )|_{E_j}. 
\end{equation}

Now the existence of $dV[\Psi]_Y$ is provided by Lemma~\ref{demu}. 
Since $d\nu$ has the property of  putting no mass on closed analytic subsets (cf. \cite[Prop. 3.3]{KS20}), so does its direct image $dV[\Psi]_Y$.  (See \cite[(1.3)]{BBEGZ} for more  on this property.)
\end{proof} 

\begin{lemma}\label{demu}
The Ohsawa measure  $dV[\Psi]_Y$ is equal to the direct image measure $ f_* d\nu = \sum^m_{j=1} f_* (d\nu|_{E_j})$ of the measure $d\nu$. 

\end{lemma}

\begin{proof}
Since $f: X' \to X$ is modification, from basic properties of integration of top degree forms, we have the equality of the integrals (inside the limit) on the RHSs of \eqref{dvp} and \eqref{dvp1}.  Hence the direct image $f_* d\nu$ of $d\nu$ under $f = f|_T : T \to Y$ satisfies the condition \eqref{dvp}. By Definition~\ref{iber} (and its uniqueness),   it is equal to the Ohsawa measure, i.e. we have $dV[\Psi]_Y = f_* d\nu$. 
\end{proof}

\begin{remark1}

Lemma~\ref{demu} was pointed out by Demailly in this generality in a discussion with the author, although only the case $m=1$ is treated in \cite[Prop. 4.5]{D15}. 

\end{remark1}

\begin{proposition}\label{recover}

The full Ohsawa measure $dV[\Psi]$ defined on the non-klt locus $Y:= N(\Psi)$ of $(X, \Psi)$ is recovered from the collection of its `restrictions'  $dV[\Psi]_{Z}$  to each irreducible component $Z$ of the non-klt locus.

\end{proposition}

\begin{proof}

Here, by  the restriction of $dV[\Psi]$ (which is a measure defined on $\yr$)  to $Z$, we mean its restriction (as a measure) to a nonempty Zariski open subset $Z'$. This restriction is nothing but $dV[\Psi]_{Z}$ as is easily seen from comparing Definition~\ref{iber} and Definition~\ref{Omeasure}.      (Note that, thanks to the property of  putting no mass on closed analytic subsets, $dV[\Psi]_{Z}$  is determined already by its restriction to a nonempty Zariski open subset of $Z$.)   
It then only remains to observe the following:  the regular locus of $Y$   can be written as the disjoint union 

\begin{equation}\label{disjo}
 \yr = \bigcup_{j \in J} \; Y'_j
\end{equation} 
of complex manifolds (of possibly different dimensions) where each complex manifold $Y'_j$ is the Zariski open subset of the irreducible component $Y_j$ of (a connected component of) $Y$ determined by the relation

$$ Y_j \setminus Y_j' = (Y_j)_{\sing} \; \cup ( {\displaystyle \cup_{k \neq j}} Y_k \cap Y_j ).  $$ \end{proof}

\begin{example}

When applying Theorem~\ref{main1} to Example~\ref{xyz},  the only singularity of the Ohsawa measure $dV[\Psi]$ occurs at $p = Y \cap Z$. If $s$ is $L^2$ with respect to $dV[\Psi]$ in a neighborhood of $p$, then it should vanish at $p$.

\end{example}

\subsection{$L^2$ extension theorems}

We first recall from \cite{D15} the following $L^2$ extension theorem  of Demailly  for log canonical pairs.
We will then derive a version, Theorem~\ref{extension} which extends sections from each component of the non-klt locus.

\begin{theorem}\label{extension0} \cite[Thm. 2.8 and Remark 2.9 (b)]{D15}
Let $(X, \omega)$ be a weakly pseudoconvex K\"ahler manifold. 	
	 Let $(X, \psi)$ be an lc pair associated to an adjoint $\QQ$-line bundle $K_X + L$ as in Definition~\ref{Psi} (so that $\psi$ is a psh metric for $L$). 	 
	 Let $h$ be a  $C^{\infty}$ hermitian metric of $L$.  Let $\Psi$ be the quasi-psh function defined by the relation  $ e^{-\psi} = h e^{-\Psi}$.  Assume that, for some $\delta > 0$,

\begin{equation}\label{curvature}
   i \Theta(L, h) +  \alpha \iddb \Psi =   i \Theta (L,\psi) +  (\alpha -1) \iddb \Psi  \ge 0
\end{equation}

\noi for all $\alpha \in [1, 1+\delta]$. 
 Let $B$ be a $\QQ$-line bundle on $X$ such that $K_X + L + B$ a $\ZZ$-line bundle. Let $b$ be a psh metric of $B$. 

Let $W$ be the non-klt locus of $(X, \psi)$. 	 
	  If we have
  a holomorphic section $s \in H^0 (\Wr, (K_X + L + B)|_{\Wr})$ satisfying 

\begin{equation}\label{input00}
 \int_{\Wr} \abs{s}_{\omega, h, b}^2  dV[\Psi] < \infty ,
\end{equation}

\noi  then there exists a holomorphic section $\wt{s} \in H^0 ( X, K_X + L + B) $ such that we have $\tilde{s}|_{\Wr} = s$ and moreover

\begin{equation}\label{output00}
  \int_X   \abs{\wt{s}}_{\omega, h, b}^2  \gamma (\delta \Psi) e^{-\Psi}  dV_{\omega}        \le   \frac{34}{\delta}  \int_{\Wr} \abs{s}_{\omega, h, b}^2  dV[\Psi].  
 \end{equation}

\end{theorem}

\noindent  Here the Ohsawa measure $dV[\Psi]$ is taken with respect to the smooth volume form $dV_\omega$ associated to $\omega$.  Also $\gamma (x) = e^{- \frac{1}{2} x}$ for $x \ge 0$ and $\gamma(x) = \frac{1}{1+x^2}$ for $x \le 0$, cf. \cite[(2.7)]{D15}.

\begin{remark1} \label{glc}
 In \cite[Thm. 2.8]{D15}, the function $\Psi$ is assumed to have `log canonical singularities along'
 the non-klt locus $Y$, which is in fact the same as the pair $(X, \Psi)$ being log canonical in the standard terminology. 
  
 \end{remark1}

We derive from \cite[Thm. 2.8]{D15} the following  $L^2$ extension theorem formulated for a  maximal lc center with a unique lc place.

\begin{theorem}\label{extension}

Let $(X, \psi)$ be an lc pair where $(X, \omega)$ is a weakly pseudoconvex K\"ahler manifold. Let $K_X + L + B$ and $\Psi, h, b$ be as in the setting of Theorem~\ref{extension0}.  

Let $Y$ be a maximal lc center of $(X, \psi)$ with a unique lc place. Then the same statement of $L^2$ extension as in Theorem~\ref{extension0} holds for a holomorphic section $s \in H^0 (Y, (K_X + L + B)|_Y)$ 
 only replacing the Ohsawa measure $dV[\Psi]$ by $dV[\Psi]_Y$. More precisely, if we have 
  a holomorphic section $s \in H^0 (Y, (K_X + L + B)|_Y)$ satisfying 

\begin{equation}\label{input1}
 \int_{\yr} \abs{s}_{\omega, h, b}^2  dV[\Psi]_Y < \infty ,
\end{equation}

\noi  then there exists a holomorphic section $\wt{s} \in H^0 ( X, K_X + L + B) $ such that we have $\tilde{s}|_Y = s$ and moreover

\begin{equation}\label{output1}
 \int_X   \abs{\wt{s}}_{\omega, h, b}^2  \gamma (\delta \Psi) e^{-\Psi}  dV_{\omega}        \le   \frac{34}{\delta}  \int_{\yr} \abs{s}_{\omega, h, b}^2  dV[\Psi]_Y.  
 \end{equation}
 
\noindent  Here the Ohsawa measure $dV[\Psi]_Y$ is taken with respect to the smooth volume form $dV_\omega$.

\end{theorem}

 \begin{proof}

	Note that $Y$ is an irreducible component of the non-klt locus $N(\psi) =: W$ of $(X, \psi)$.  Given the section $s$ on $Y$ satisfying \eqref{input1}, we  will  extend it first to the entire $W$ by zero outside $Y$ as follows. Since this zero extension to $W$ is of local nature, we may write $ W:= Y_1 \cup \ldots \cup Y_m$ where $Y_1, \ldots, Y_m$ are the maximal lc centers of $(X, \psi)$ with  $Y = Y_1$. 
	
From \eqref{input1}, we see that the section $s$ (itself) is locally $L^2$ with respect to  $dV[\Psi]_Y$ on $\yr$ since $b$ is psh. 
	We need to show that $s$ is zero on every intersection $Y \cap Y_j$ for $j = 2, \ldots, m$. Fix one such $j$. Note that $Y \cap Y_j$ may not be irreducible or even connected.   By Proposition~\ref{basic}, every irreducible component $Z$ of $Y \cap Y_j$ is an lc center of $(X, \psi)$. Then by Theorem~\ref{main1}, $s$ vanishes along every such $Z$.  Therefore we can extend $s \in H^0 (Y, (K_X + L + B)|_Y)$ to a section $s' \in H^0 (W, (K_X + L + B)|_W)$ to be zero on $W \setminus Y$. 
	
	Then it is clear that $s'$ has finite $L^2$ norm with respect to the full Ohsawa measure ${dV[\Psi]}$ on $W$.  Thus we can now apply Theorem~\ref{extension0} to extend $s'$ to a section $\wt{s}$ on $X$ as in  the conclusion.  
\end{proof}

\begin{remark1}

 The statement of Theorem~\ref{extension} is formulated so that, in particular, it is directly comparable to \cite[Thm. 4.2]{K07}. In particular, it is arranged to  indicate the necessary semipositivity of the line bundle $L$ (in \eqref{curvature})  for the  $L^2$ extension to hold. 
 The semipositivity  required on $B$ is minimal : it can be taken to be $\OO_X$ if $K_X + L$ is a $\ZZ$-line bundle.

\end{remark1}

 We have the following application of Theorem~\ref{extension} which also illustrates use of Theorem~\ref{main1}. 

\begin{prop}[cf. Proposition~\ref{basic} (3)]\label{normality}
Let  $(X, \psi)$ be an lc pair. 
Let $p \in X$ be a point in the non-klt locus of $(X, \psi)$. Let $Y$ be the lc center of $(X, \psi)$ that is minimal  at $p$, i.e. minimal with respect to inclusion among those lc centers passing through $p$. Assume that $Y$ is also a maximal lc center of $(X, \psi)$ with a unique lc place.  Then $Y$ is normal at $p$. 

\end{prop}

\begin{proof} 

 The existence and uniqueness of $Y$ was given in Proposition~\ref{basic} (3).  Consider the Ohsawa measure $dV[\Psi]_Y$ (for a choice of $\Psi$ and $dV_X$). 
Since $Y$ is minimal at $p$, there exists a Zariski closed subset $W \subset X$ such that $Y \setminus W$ is a minimal lc center for $(X \setminus W, \psi|_{X \setminus W})$.

Let $Y' := Y \setminus W$ and $X' := X \setminus W$.  Let $\widetilde{\OO}_{Y'}$ be the sheaf of weakly holomorphic functions on $Y'$, i.e. holomorphic functions $f$ on $(Y')_{\reg}$ such that every point of $(Y')_{\sing}$ has a neighborhood $V$ so that $f$ is bounded on $(Y')_{\reg} \cap V$. We need to show that 
 $\widetilde{\OO}_{Y'} = \OO_{Y'}$ by \cite[Chap.II \S 7]{DX}.

 By Corollary~\ref{coro2} (2), the Ohsawa measure $dV[\Psi]_Y$ is locally integrable when restricted to $Y'$. Hence every weakly holomorphic function germ on $Y'$ is locally $L^2$ with respect to $dV[\Psi]_{Y'}$ at every point $q \in Y'$. Applying Theorem~\ref{extension} to a Stein neighborhood $V \subset X'$  of $q$, $f$ can be extended from $V \cap (Y')_{\reg}$ to a holomorphic function $F$ on $V$. The restriction of $F$ to $Y' \cap V$ is holomorphic, hence  we have $\widetilde{\OO}_{Y'} = \OO_{Y'}$. 
\end{proof}

 \medskip

\section{Proofs of the main results}  

Based on preparations from previous sections, we complete the proofs of the main results. 
  First we revisit the following fact  from \cite[Prop. 4.5]{D15}, whose proof we treat here in detail in our setting together wth the necessary assumption of `unique lc place' made explicit.

\begin{proposition}\label{thm4}

 Let $Y$ be a maximal lc center of $(X, \Psi)$ as in the setting of Definition~\ref{iber}. 
  If $Y$ has a unique lc place, then the Ohsawa measure $dV[\Psi]_Y$ has smooth positive density with respect to the Lebesgue measure on a Zariski open set of $\yr$.

\end{proposition}

 \begin{proof} 
 
   We continue to use the setting and notation of the proof of Proposition~\ref{exist}. In particular, let $f: X' \to X$ be a log resolution of the pair $(X, \Psi)$. 
  Let $Y_0$ be the Zariski open subset of $\yr$ defined by $$Y_0 = \yr \setminus (Y_2 \cup \ldots \cup Y_m \cup Z_1 \cup \ldots \cup Z_k ) $$ where $Y_2, \ldots, Y_m$ are other maximal lc centers than $Y:=Y_1$ (we may assume that there are only a finite number of them by restricting to a neighborhood of $Y$)  and $Z_1, \ldots, Z_k$ are lc centers of $(X, \Psi)$ that are properly contained in $Y$.
  
  Let $E \subset X'$ be the unique lc place.  From the proof of Proposition~\ref{exist}, the Ohsawa measure $dV[\Psi]_Y$ is the direct image along $f = f|_E : E \to Y$ of a measure $d\nu$ on $E$ with poles along the divisor $R$ from \eqref{pole} which simplifies to $R = F |_E$ since $E$ is the unique lc place of $Y$.  The coefficients of the prime divisors $G$ in $R = F|_E$ are at most $1$.  When the coefficient of $G \subset X'$ is equal to $1$, its image in $Y$ is contained in the complement of $Y_0$. Hence when restricted to $Y_0$, the Ohsawa measure is the direct image of a locally integrable volume form on $f^{-1} (Y_0)$, which concludes the proof. 
 \end{proof}

 \begin{proof}[\textbf {Proof of Theorem~\ref{infinite}}]
 
  Let $E_1, \ldots, E_m$ be the lc places with $m \ge 2$ and let $T = E_1 \cup \ldots \cup E_m$.
 Note that for any $i, j$, we have $f(E_i \cap E_j) \subset Y$.  We first claim that 
  there exist (at least) two different lc places, say $E_1$ and $E_2$, such that $f(E_1 \cap E_2) = Y$.

 Suppose otherwise. Then we have a finite number of Zariski closed proper subsets of $Y$ given by $f(E_i \cap E_j)$ for every $1 \le i,j \le m$. Hence  there exists $p \in Y$ such that $p$ does not belong to any of $f(E_i \cap E_j)$.

 On the other hand, since  the restriction of $f$ to $T$,  $f|_T : T \to Y$ has connected fibers by Proposition~\ref{basic0}, the set $$f^{-1} (p) \cap T = (f^{-1} (p) \cap E_1) \cup \ldots \cup (f^{-1} (p) \cap E_m)$$ is connected. Hence there exist $i \neq j$ such that  $(f^{-1} (p) \cap E_i)) \cap (f^{-1} (p) \cap E_j) \neq \emptyset$. Then we have $p \in f(E_i \cap E_j)$, which is contradiction. The above claim is proved.

  The remaining part of the argument was first pointed out by Chen-Yu Chi~\cite{C20} (possibly in different terms).    
   From Lemma~\ref{demu},  the Ohsawa measure is equal to the sum of the measures $f_* (d\nu)|_{E_j}$ where $(d\nu)|_{E_j}$ refers to the restriction to the open set $E_j'$ in \eqref{ejp} in the proof of Proposition~\ref{exist}.  When $j=1$, $(d\nu)|_{E_j}$ is equal to the infinity measure, hence so are the direct image $f_* (d\nu)|_{E_j}$ and  the Ohsawa measure. 
 \end{proof}

 \begin{remark1}
 
 The dichotomy between Theorem~\ref{infinite} and Proposition~\ref{thm4} reconfirms Proposition~\ref{elc} since the definition of the Ohsawa measure is independent of log resolutions.

 \end{remark1}

\begin{remark1}\label{MD}

 The phenomenon of the infinity measure in Theorem~~\ref{infinite} does not occur in the classical case of \cite{M93}, \cite{D00}: see \cite[(2.6)]{D15}. 

\end{remark1}

 We now give the proof of Theorem~\ref{main1}. 

\begin{proof}[\textbf{Proof of Theorem~\ref{main1}}]

Let $p \in Y$ be a point. For the purpose of examining the property of locally $L^2$ at $p$,  we may only consider the   finite number of irreducible components $Y_1, \ldots, Y_m$ each containing $p$, of the non-klt locus of $(X, \Psi)$.  Depending on $p$, we have one of the following three cases.

\begin{itemize}

\item $m=1$ and $Y_1$ has non-unique lc places.

\item $m=1$ and  $Y_1$ has a unique lc place. 

\item $m \ge 2$, i.e. $p \in Y_1 \cup \ldots \cup Y_m$ and some of $Y_1, \ldots, Y_m$ have non-unique lc places. 

\end{itemize}

   In the first case, $s$ is locally $L^2$ at $p$ if and only if $s$ is identically zero on $Y_j$ due to  Theorem~\ref{infinite}. In the third case, we may restrict our attention to each component $Y_j$ by considering $s|_{Y_j}$ with respect to $dV[\Psi]_{Y_j}$. 
Therefore, in order to prove  the theorem, it suffices to assume that the non-klt locus is irreducible and equal to $Y=Y_1$ with a unique lc place.

	From the proof of Proposition~\ref{exist}, the Ohsawa measure $dV[\Psi] = dV[\Psi]_Y$ is the direct image along $f = f|_E : E \to Y$ of a measure $d\nu$ on $E$ with poles along the divisor $R = F |_E$ from \eqref{pole} whose coefficients are at most $1$.  By the definition of the direct image of a measure, for every open subset $U \subset Y$, we have (where $f^{-1} (U) \subset E$) 
	
	\begin{equation}\label{s1}
	\displaystyle \int_U \abs{s}^2 dV[\Psi] = \int_{f^{-1} (U)} \abs{f^*s}^2  d\nu .
	\end{equation}

	First, suppose the condition (2). If a prime divisor $Q$ in the above $F|_E$ appears with coefficient $1$, it must be coming from a properly contained lc center $C_Q$ in $Y$ (since $Y$ itself has a unique lc place). Since $s$ vanishes along $C_Q$, $f^* s$ vanishes along $Q$. This means that the RHS of \eqref{s1} is locally finite, hence gives the condition (1) of the assertion. 
	
 Now suppose the condition (1) that $s$ is locally $L^2$ at every point of $Y$ with respect to $dV[\Psi]$.  
  Let $C$ be a non-maximal lc center contained in $Y$.  We may assume that $C$ is not contained in another such lc center $C' \subset Y$ for the purpose of showing that $s$ vanishes along $C$.  Let $E^1_C, \ldots, E^m_C$ be the lc places of $C$ in the given log resolution.

\begin{lemma}\label{one}
There exists at least one $i$ such that $f(E^i_C \cap E) = C$. 

\end{lemma}
 
 \begin{proof}[Proof of Lemma~\ref{one}]
 
    Suppose otherwise. Consider all the closed subsets given by $f(E^i_C \cap E)$ for $1 \le i \le m$ and given by lc centers contained in $C$. Since these are all properly contained, we may choose $p \in C$ in the complement. We will apply Proposition~\ref{basic0} (2) to the set of lc centers $\{ Y, C \}$ where $W = Y \cup C = Y$. Consider $f: T \to Y$ : note that $T$ may contain prime divisors whose images are properly contained lc centers in $C$. By restricting to a neighborhood of $p$, we may assume that $T$ consists of prime divisors whose images are either  $C$ or $Y$ : these are precisely $E^1_C, \ldots, E^m_C$ and $E$. 
  By Proposition~\ref{basic0} (2), we see that $f^{-1} (p) \cap (E^1_C \cup  \ldots \cup E^m_C \cup E)$ is connected. 
 Hence at least for one $i$, we have $f^{-1} (p) \cap E^i_C $ has nonempty intersection with $f^{-1} (p) \cap E$.  Then we have $p \in f(E^i_C \cap E)$, which is contradiction.  
 \end{proof}
 
 From \eqref{s1}, $f^* s$ vanishes along $E^k_C |_E$ for every lc place $E^k_C$ of $C$.  By the above lemma,  there exists $i$ such that $f(E^i_C \cap E) = f(E_C) = C$. Hence $s$ vanishes along $C$. This completes the proof of Theorem~\ref{main1}. 
\end{proof}

 We have the following consequences from Theorem~\ref{main1}. 

\begin{corollary}\label{coro2}

Let $(X, \Psi)$ and $Y$ be as in Theorem~\ref{main1}. 

(1) 
If a holomorphic function $s$ on $Y$ is locally $L^2$ with respect to ${dV[\Psi]}$, then $s$ vanishes along all the intersections of (at least two) irreducible components of the non-klt locus of $(X, \Psi)$. 

(2)   The Ohsawa measure ${dV[\Psi]}$ is locally integrable if and only  if each connected component $Z$ of $Y$ is irreducible and $Z$ is a   minimal  log canonical center of $(X, \Psi)$ with a unique log canonical place. 

\end{corollary}

\begin{proof}

(1) The irreducible components are maximal lc centers. Since the intersection of a set of lc centers is written as a union of lc centers (Proposition~\ref{basic}), the assertion follows from Theorem~\ref{main1}. 

(2) Clearly one may assume that $Y$ is connected (so $Z= Y$). By definition, $dV[\Psi]$ is locally integrable if and only if  the constant function $s = 1$ is locally integrable with respect to $dV[\Psi]$. 

If this holds, due to the condition (2) of Theorem~\ref{main1},  $Y$ cannot be reducible and $Y$ cannot contain any non-maximal lc center. This gives one direction of the statement and the converse is also clear for the same reason. 
\end{proof}

\begin{proof}[\textbf{Proof of Corollary~\ref{coro3}}]

Suppose that $s$ satisfies (1). Then as in the proof of Theorem~\ref{extension}, $s$ can be extended to the entire non-klt locus by setting it as zero on all the other irreducible components than $Y$. Then applying Theorem~\ref{main1}, we see that $s$ should vanish along all non-maximal lc centers contained in $Y$. 

Now suppose that $s$ satisfies (2). If any other lc center $Z$ intersects with $Y$, then $Z \cap Y$ is a union of lc centers of the same pair by Proposition~\ref{basic} (2). Hence $s$ vanishes along $Z \cap Y$ and this means that we can extend $s$ to the entire non-klt locus by setting it as zero outside $Y$. By Theorem~\ref{main1}, $s$ is locally $L^2$ with respect to $dV[\Psi]$, hence also with respect to $dV[\Psi]_Y$ on $Y$. 
\end{proof}

\qa

\qa

\normalsize

\noi \textsc{Dano Kim}

\noi Department of Mathematical Sciences and Research Institute of Mathematics

\noi Seoul National University, 08826  Seoul, Korea

\noi Email address: kimdano@snu.ac.kr

\end{document}